\documentclass[12pt]{article}
\usepackage{amssymb}
\usepackage{amsmath}
\usepackage{amsthm}
\usepackage{geometry}		

\let\originalleft\left
\let\originalright\right
\renewcommand{\left}{\mathopen{}\mathclose\bgroup\originalleft}
\renewcommand{\right}{\aftergroup\egroup\originalright}

\begin{document}

\newtheorem{theorem}{Theorem}
\newtheorem{lemma}[theorem]{Lemma}
\theoremstyle{definition}
\newtheorem{definition}{Definition}


\title{
A piecewise-linear fixed point theorem.
}
\author{
D.J.W.~Simpson\\\\
School of Mathematical and Computational Sciences\\
Massey University\\
Palmerston North, 4410\\
New Zealand
}
\maketitle



\begin{abstract}

We prove that if a continuous piecewise-smooth map on $\mathbb{R}^n$ is comprised of two linear functions,
has a bounded orbit, and satisfies a certain non-degeneracy condition, then it has a fixed point.
The result has important consequences to the bifurcation theory of nonsmooth dynamical systems,
yet the proof requires only elementary linear algebra.

\end{abstract}

\subsection*{Fixed point theorems}

A {\em fixed point} of a map $f$ is a point $x^*$ for which $f(x^*) = x^*$.
The existence of a fixed point can be assured in many ways
leading to a wide variety of fixed point theorems.
For instance the Banach fixed point theorem uses contractiveness,
the Lefschetz fixed point theorem uses algebraic topology and index theory,
while the Brouwer fixed point theorem simply uses the domain of the map \cite{GrDu03,Br71}.
The piecewise-linear fixed point theorem introduced below refers to the dynamics of the map,
another example of this being that if an orientation-preserving homeomorphism on $\mathbb{R}^2$ has a periodic solution, then it must have a fixed point \cite{Le08}.
More generally maps on $\mathbb{R}^n$ can have periodic solutions and other interesting dynamics without having fixed points \cite{JiLi16}.

\subsection*{A piecewise-linear fixed point theorem}

Let $A \in \mathbb{R}^{n \times n}$ ($n \ge 2$), $b, c \in \mathbb{R}^n$, and consider
\begin{equation}
f(x) = A x + b|x_1| + c,
\label{eq:f}
\end{equation}
where $x_1$ denotes the first component of $x \in \mathbb{R}^n$.
This is a piecewise-linear map on $\mathbb{R}^n$;
in fact any continuous, two-piece, piecewise-linear map on $\mathbb{R}^n$ can be put in the form \eqref{eq:f}.
In what follows we write $f^k(x)$ for the $k^{\rm th}$ iterate of $x$ under $f$,
and $I$ for the $n \times n$ identity matrix.

\begin{theorem}
Let $P \in \mathbb{R}^{n \times (n-1)}$ denote second through $n^{\rm th}$ columns of $I - A$.
If $P$ has full rank and there exists $x \in \mathbb{R}^n$ such that $\| f^k(x) \| \not\to \infty$ as $k \to \infty$,
then $f$ has a fixed point in $\mathbb{R}^n$.
\label{th:main}
\end{theorem}

The requirement that $P$ has full rank (i.e.~its columns are linearly independent) is our non-degeneracy condition.
This condition fails only on a codimension-two subspace of the space of all such maps \eqref{eq:f}.
To demonstrate its necessity, consider
\begin{equation}
A = \begin{bmatrix}
-\frac{1}{2} & 1 & 0 \\[1mm]
-\frac{1}{2} & 0 & 0 \\[1mm]
-\frac{11}{28} & 0 & 1
\end{bmatrix}, \quad
b = \begin{bmatrix} -\frac{1}{2} \\[1mm] -1 \\ \frac{3}{28} \end{bmatrix}, \quad
c = \begin{bmatrix} 1 \\ 0 \\ 0 \end{bmatrix},
\label{eq:counterExample}
\end{equation}
with which
\begin{equation}
P = \begin{bmatrix} -1 & 0 \\ 1 & 0 \\ 0 & 0 \end{bmatrix}
\nonumber
\end{equation}
does not have full rank.
With \eqref{eq:counterExample} the map $f$ has no fixed point
(it is readily seen $f(x) = x$ yields no solutions),
yet  for any $s \in \mathbb{R}$,
\begin{align*}
x = \begin{bmatrix} -\frac{2}{15} \\[1mm] -\frac{7}{5} \\[1mm] s \end{bmatrix}, \quad
f(x) = \begin{bmatrix} -\frac{2}{5} \\[1mm] -\frac{1}{15} \\[1mm] s + \frac{1}{15} \end{bmatrix}, \quad
f^2(x) = \begin{bmatrix} \frac{14}{15} \\[1mm] -\frac{1}{5} \\[1mm] s + \frac{4}{15} \end{bmatrix}, \quad
f^3(x) = x,
\end{align*}
so certainly $\| f^k(x) \| \not\to \infty$.

\subsection*{Fixed points and matrix adjugates}

Before we can prove Theorem \ref{th:main} we need some preliminary calculations.
To this end, we write \eqref{eq:f} in the explicitly piecewise form
\begin{equation}
f(x) = \begin{cases}
\left( A - b e_1^{\sf T} \right) x + c, & x_1 \le 0, \\
\left( A + b e_1^{\sf T} \right) x + c, & x_1 \ge 0,
\end{cases}
\label{eq:f2}
\end{equation}
where $e_1^{\sf T} = \begin{bmatrix} 1 & 0 & \cdots & 0 \end{bmatrix}$; also let
\begin{equation}
M_\pm = I - A \mp b e_1^{\sf T}.
\nonumber
\end{equation}
If $\det(M_-) \ne 0$ the left piece of $f$ (i.e.~the $x_1 \le 0$ piece of \eqref{eq:f2}) has the unique fixed point
\begin{equation}
y^- = M_-^{-1} c.
\label{eq:xL}
\end{equation}
This is a fixed point of $f$ if and only if $y^-_1 \le 0$.
Similarly if $\det(M_+) \ne 0$ the right piece of $f$ has the unique fixed point
\begin{equation}
y^+ = M_+^{-1} c,
\label{eq:xR}
\end{equation}
and this is a fixed point of $f$ if and only if $y^+_1 \ge 0$.
Notice that if $\det(M_-) \ne 0$ and $\det(M_+) \ne 0$ then $y^-$ and $y^+$ are the only possible
fixed points of $f$.
In order to algebraically connect the values $y^-_1$ and $y^+_1$ we use matrix adjugates \cite{Be92,Ko96}:

\begin{definition}
Let $M \in \mathbb{R}^{n \times n}$ be a matrix.
For each $i,j = 1,\ldots,n$
let $m_{ij}$ be the determinant of the $(n-1) \times (n-1)$ matrix formed by
removing the $i^{\rm th}$ row and $j^{\rm th}$ column from $M$ (the $m_{ij}$ are the {\em minors} of $M$).
The {\em adjugate} of $M$, denoted ${\rm adj}(M)$, is the $n \times n$ matrix
with $(i,j)$-entry $(-1)^{i+j} m_{ji}$ for all $i, j = 1,\ldots,n$.
\label{df:adjugate}
\end{definition}

The key property of the matrix adjugate is that any $M \in \mathbb{R}^{n \times n}$ obeys
\begin{equation}
{\rm adj}(M) M = \det(M) I,
\label{eq:adjIdentity}
\nonumber
\end{equation}
so if $\det(M) \ne 0$ then $M^{-1} = \frac{{\rm adj}(M)}{\det(M)}$.
Now let
\begin{equation}
u^{\sf T} = e_1^{\sf T} {\rm adj}(I - A),
\label{eq:u}
\end{equation}
and notice
\begin{equation}
u^{\sf T} = e_1^{\sf T} {\rm adj}(M_\pm),
\nonumber
\end{equation}
because $u$ is independent of the first column of $M_\pm$.
Thus by \eqref{eq:xL} and \eqref{eq:xR},
\begin{equation}
y^\pm_1 = e_1^{\sf T} M_\pm^{-1} c
= \frac{u^{\sf T} c}{\det(M_\pm)}.
\label{eq:xLR1}
\end{equation}
This formula connects $y^-_1$ and $y^+_1$ and we use it below to prove Theorem \ref{th:main}.

\subsection*{Main arguments}

\begin{lemma}
If $P$ has full rank and $u^{\sf T} c = 0$ then $f$ has a fixed point.
\label{le:uTc}
\end{lemma}

\begin{proof}
Suppose $P$ has full rank.
Then there exists $i \in \{ 1,\ldots,n \}$ such that removing the $i^{\rm th}$ row from $P$
leaves an invertible $(n-1) \times (n-1)$ matrix,
thus the $i^{\rm th}$ component of $u$ is non-zero, so certainly $u \ne 0$.

Suppose also $u^{\sf T} c = 0$.
Then $\det(M_-) = 0$, for otherwise $y^-$, given by \eqref{eq:xL}, would be a fixed point of $f$ (with $y^-_1 = 0$).
The nullspace of $M_-$ cannot be more than one-dimensional, for otherwise ${\rm adj}(M_-)$ would
be the zero matrix and then $u = 0$.
So the nullspace of $M_-$ is one-dimensional, spanned by some non-zero $v \in \mathbb{R}^n$.
Notice $v_1 \ne 0$, for otherwise the remaining components of $v$ would form a non-zero vector in the nullspace of $P$
(which we know has full rank).
Also $u$ spans the nullspace of $M_-^{\sf T}$,
because $u^{\sf T} M_- = e_1^{\sf T} {\rm adj}(M_-) M_- = \det(M_-) e_1^{\sf T} = 0$.
But $u^{\sf T} c = 0$, thus $c$ belongs to the column space of $M_-$
by the fundamental theorem of linear algebra.
That is, $M_- z = c$ for some $z \in \mathbb{R}^n$,
and so $M_- (z + t v) = c$ for all $t \in \mathbb{R}$.
Since $v_1 \ne 0$ there exists $t \in \mathbb{R}$ such that $z_1 + t v_1 \le 0$,
with which $z + t v$ is a fixed point of $f$.
\end{proof}

\begin{proof}[Proof of Theorem \ref{th:main}]
Suppose $P$ has full rank and $f$ has no fixed points.
To prove Theorem \ref{th:main} (by contrapositive) we show every forward orbit of $f$ diverges.

By Lemma \ref{le:uTc}, $u^{\sf T} c \ne 0$.
In view of the substitution $x \mapsto -x$ we can assume $u^{\sf T} c > 0$.
Then \eqref{eq:xLR1} implies $\det(M_-) \ge 0$, for otherwise $y^-$ is a fixed point of $f$;
also $\det(M_+) \le 0$, for otherwise $y^+$ is a fixed point of $f$.
Then for any $x \in \mathbb{R}^n$ with $x_1 \le 0$,
\begin{equation}
u^{\sf T} f(x) = u^{\sf T} \left( \left( A - b e_1^{\sf T} \right) x + c \right)
= u^{\sf T} (x - M_- x + c).
\nonumber
\end{equation}
But $u^{\sf T} M_- x = e_1^{\sf T} {\rm adj}(M_-) M_- x = \det(M_-) x_1$,
so since $\det(M_-) \ge 0$ and $x_1 \le 0$ we have
\begin{equation}
u^{\sf T} f(x) \ge u^{\sf T} x + u^{\sf T} c.
\label{eq:proof2}
\end{equation}
Similarly for any $x \in \mathbb{R}^n$ with $x_1 \ge 0$,
\begin{equation}
u^{\sf T} f(x) = u^{\sf T} (x - M_+ x + c),
\nonumber
\end{equation}
where $u^{\sf T} M_+ x = \det(M_+) x_1$, so again we have \eqref{eq:proof2}.
Thus for any forward orbit of $f$
the value of $u^{\sf T} x$ increases by at least $u^{\sf T} c > 0$ every iteration,
so the orbit diverges.
\end{proof}

\subsection*{Consequences to bifurcation theory}

As the parameters of a piecewise-smooth map are varied,
a {\em border-collision bifurcation} (BCB) occurs when a fixed point collides with a switching manifold \cite{DiBu08,Si16}.
BCBs correspond to the onset of sticking motion in mechanical systems with stick-slip friction \cite{DiKo03},
changes in the firing patterns of neurons \cite{GrKr14} and in the periods of trade cycles \cite{PuGa06},
and create chaotic dynamics in various electrically-switched control systems \cite{YuBa98}.
The local dynamics associated with many BCBs can be characterised by a piecewise-linear map of the form \eqref{eq:f},
where $u^{\sf T} c$ changes sign as we pass through the bifurcation.

Since at least the 1963 work of Brousin {\em et al.}~\cite{BrNe63},
it has been known there are two generic cases for the behaviour of fixed points near BCBs.
By \eqref{eq:xLR1}, if $\det(M_-) \det(M_+) > 0$ the map has one fixed point on each side of the bifurcation,
while if $\det(M_-) \det(M_+) < 0$ the map has two fixed points on one side of the bifurcation
and no fixed points on the other side of the bifurcation (this case is termed a {\em nonsmooth fold}).
Theorem \ref{th:main} applies to nonsmooth folds and shows, at least for leading-order truncated form,
there are no local invariant sets on one side of the bifurcation.
That the map simply has no period-two solutions on this side of the bifurcation was conjectured by Feigin in the 1970's \cite{Fe78},
but only proven relatively recently \cite{Si14d}.

Theorem \ref{th:main} was inspired by Proposition 19 of Carmona {\em et al.}~\cite{CaFr02}
which is a similar result for piecewise-linear differential equations.
It remains to generalise the result to other classes of nonsmooth dynamical systems,
particularly Filippov systems and hybrid systems that also exhibit nonsmooth folds \cite{DiNo08}.
This would show that nonsmooth folds where one fixed point is attracting bring about a catastrophic change to the attractor,
i.e.~such bifurcations can be viewed as tipping points \cite{BuGr22}.

\subsection*{Acknowledgements}

This work was supported by Marsden Fund contract MAU2209 managed by Royal Society Te Ap\={a}rangi.


\begin{thebibliography}{10}

\bibitem{GrDu03}
A.~Granas and J.~Dugundji.
\newblock {\em Fixed Point Theory.}
\newblock Springer, New York, 2003.

\bibitem{Br71}
R.F. Brown.
\newblock {\em The {L}efschetz Fixed Point Theorem.}
\newblock Scott, Foresman \& Co., Glenview, IL, 1971.

\bibitem{Le08}
P.~Le~Calvez.
\newblock Pourquoi les points p\'{e}riodiques des hom\'{e}omorphismes du plan
  tournent-ils autour de certains points fixes?
\newblock {\em Ann. Scient. \'{E}c. Norm. Sup.}, 41:141--176, 2008.
\newblock In French.

\bibitem{JiLi16}
H.~Jiang, Y.~Liu, Z.~Wei, and L.~Zhang.
\newblock Hidden chaotic attractors in a class of two-dimensional maps.
\newblock {\em Nonlinear Dyn.}, 85:2719--2727, 2016.

\bibitem{Be92}
S.K. Berberian.
\newblock {\em Linear Algebra.}
\newblock Oxford University Press, New York, 1992.

\bibitem{Ko96}
B.~Kolman.
\newblock {\em Elementary Linear Algebra.}
\newblock Prentice Hall, Upper Saddle River, NJ, 1996.

\bibitem{DiBu08}
M.~di~Bernardo, C.J. Budd, A.R. Champneys, and P.~Kowalczyk.
\newblock {\em Piecewise-smooth Dynamical Systems. Theory and Applications.}
\newblock Springer-Verlag, New York, 2008.

\bibitem{Si16}
D.J.W. Simpson.
\newblock Border-collision bifurcations in $\mathbb{R}^n$.
\newblock {\em SIAM Rev.}, 58(2):177--226, 2016.

\bibitem{DiKo03}
M.~di~Bernardo, P.~Kowalczyk, and A.~Nordmark.
\newblock Sliding bifurcations: {A} novel mechanism for the sudden onset of
  chaos in dry friction oscillators.
\newblock {\em Int. J. Bifurcation Chaos}, 13(10):2935--2948, 2003.

\bibitem{GrKr14}
A.~Granados, M.~Krupa, and F.~Cl\'{e}ment.
\newblock Border collision bifurcations of stroboscopic maps in periodically
  driven spiking models.
\newblock {\em SIAM J. Appl. Dyn. Syst.}, 13(4):1387--1416, 2014.

\bibitem{PuGa06}
T.~Puu, L.~Gardini, and I.~Sushko.
\newblock On the change of periodicities in the {H}icksian
  multiplier-accelerator model with a consumption floor.
\newblock {\em Chaos Solitons Fractals}, 29:681--696, 2006.

\bibitem{YuBa98}
G.~Yuan, S.~Banerjee, E.~Ott, and J.A. Yorke.
\newblock Border-collision bifurcations in the buck converter.
\newblock {\em IEEE Trans. Circuits Systems I Fund. Theory Appl.},
  45(7):707--716, 1998.

\bibitem{BrNe63}
V.A. Brousin, Yu.I. Neimark, and M.I. Feigin.
\newblock On some cases of dependence of periodic motions of relay system upon
  parameters.
\newblock {\em Izv. Vyssh. Uch. Zav. Radiofizika}, 4:785--800, 1963.
\newblock In Russian.

\bibitem{Fe78}
M.I. Feigin.
\newblock On the structure of {$C$}-bifurcation boundaries of
  piecewise-continuous systems.
\newblock {\em J. Appl. Math. Mech.}, 42(5):885--895, 1978.
\newblock Translation of {\em Prikl.~Mat.~Mekh.}, 42(5):820-829, 1978.

\bibitem{Si14d}
D.J.W. Simpson.
\newblock On the relative coexistence of fixed points and period-two solutions
  near border-collision bifurcations.
\newblock {\em Appl. Math. Lett.}, 38:162--167, 2014.

\bibitem{CaFr02}
V.~Carmona, E.~Freire, E.~Ponce, and F.~Torres.
\newblock On simplifying and classifying piecewise-linear systems.
\newblock {\em IEEE Trans. Circuits Systems I Fund. Theory Appl.},
  49(5):609--620, 2002.

\bibitem{DiNo08}
M.~di~Bernardo, A.~Nordmark, and G.~Olivar.
\newblock Discontinuity-induced bifurcations of equilibria in piecewise-smooth
  and impacting dynamical systems.
\newblock {\em Phys. D}, 237:119--136, 2008.

\bibitem{BuGr22}
C.~Budd, C.~Griffith, and R.~Kuske.
\newblock Dynamic tipping in the non-smooth {S}tommel-box model, with fast
  oscillatory forcing.
\newblock {\em Phys. D}, 432:132948, 2022.

\end{thebibliography}
\end{document}